\documentclass[12pt,twoside]{article}

\usepackage{fancyhdr}
\usepackage{amsmath}
\usepackage{amscd}
\usepackage{amssymb}
\usepackage{amsfonts}
\usepackage{euscript}
\usepackage{oldgerm}
\usepackage{calligra}
\usepackage{mathrsfs}
\usepackage{hyperref}
\usepackage{url}
\usepackage{lastpage}
\usepackage{multirow}
\usepackage{graphicx}
\usepackage{stmaryrd}
\usepackage{wasysym}
\usepackage{pifont}
\usepackage{diagrams}
\usepackage{draftwatermark}
\SetWatermarkAngle{45}
\SetWatermarkLightness{0.9}
\SetWatermarkFontSize{5cm}
\SetWatermarkScale{2}
\SetWatermarkText{ }

\renewcommand{\thefootnote}{\fnsymbol{footnote}}

\long\def\sfootnote[#1]#2{\begingroup
\def\thefootnote{\fnsymbol{footnote}}\footnote[#1]{#2}\endgroup}

\newtheorem{theorem}{Theorem}[section]

\newtheorem{corollary}[theorem]{Corollary}

\newenvironment{proof}{\noindent\mbox{\bf Proof.}}
{\hfill\mbox{\ding{111}}\bigskip}

\begin{document}

\pagestyle{fancy}
\lhead[page \thepage \ (of \pageref{LastPage})]{}
\chead[{\bf  Diagonalizing by Fixed--Points}]{{\bf  Diagonalizing by Fixed--Points}}
\rhead[]{page \thepage \ (of \pageref{LastPage})}
\lfoot[\copyright\ {\sf Ahmad Karimi \& Saeed Salehi 2014}]
{$\varoint^{\Sigma\alpha\epsilon\epsilon\partial}_{\Sigma\alpha\ell\epsilon\hslash\imath}
\centerdot${\footnotesize {\rm ir}}}
\cfoot[{\footnotesize {\tt  }}]{{\footnotesize {\tt  }}}
\rfoot[$\varoint^{\Sigma\alpha\epsilon\epsilon\partial}_{\Sigma\alpha\ell\epsilon\hslash\imath}\centerdot$
{\footnotesize {\rm ir}}]{\copyright\ {\sf Ahmad Karimi \& Saeed Salehi 2014}}
\renewcommand{\headrulewidth}{1pt}
\renewcommand{\footrulewidth}{1pt}
\thispagestyle{empty}

\begin{center}
\begin{table}
\hspace{0.75em}
\begin{tabular}{| c | l  || l | c |}
\hline
 \multirow{7}{*}{\includegraphics[scale=0.65]{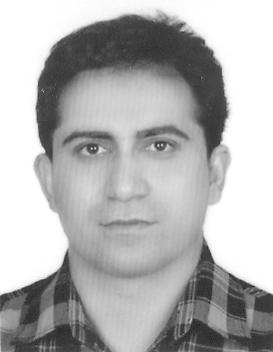}}&    &  &
 \multirow{7}{*}{ \ \ } \ \ \ \ \\
 &     \ \ {\large{\sc Ahmad Karimi}}  \ \  \ & \ \    Tel: \, +98 (0)919 510 2790     &  \\
 &   \ \ Department of  Mathematics \ \  \ & \ \ Fax: \ +98 (0)21 8288 3493    & \\
 &   \ \ Tarbiat Modares University  \ \ \  & \ \ E-mail: \!\!{\tt  a.karimi40}{\sf @}{\tt  yahoo.com}   &  \\
 &  \ \ P.O.Box  14115--134  \ \ \ &   \ \  Behbahan KA Univ. of  Tech.    &  \\
 &   \ \ Tehran, IRAN \ \ \ &  \ \   61635--151 Behbahan, IRAN &   \\
 &    &  &  \\
 \hline
\hline
 \multirow{7}{*}{\includegraphics[scale=0.4]{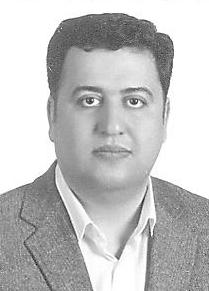}}&    &  &
 \multirow{7}{*}{${\huge \varoint^{\Sigma\alpha\epsilon\epsilon\partial}_{\Sigma\alpha\ell\epsilon\hslash\imath}\centerdot}${{\rm ir}}} \\
 &     \ \ {\large{\sc Saeed Salehi}}  \ \  \ & \ \    Tel: \, +98 (0)411 339 2905    &  \\
 &   \ \ Department of Mathematics \ \  \ & \ \ Fax: \ +98 (0)411 334 2102   & \\
 &   \ \ University of Tabriz \ \ \  & \ \ E-mail: \!\!{\tt /root}{\sf @}{\tt SaeedSalehi.ir/}   &  \\
 &  \ \ P.O.Box 51666--17766 \ \ \ &   \ \ \ \ {\tt /SalehiPour}{\sf @}{\tt TabrizU.ac.ir/}   &  \\
 &   \ \ Tabriz, IRAN \ \ \ & \ \ Web: \  \ {\tt http:\!/\!/SaeedSalehi.ir/}  &  \\
 &    &  &  \\
 \hline
\end{tabular}
\end{table}
\end{center}

\vspace{1.5em}

\begin{center}
{\bf {\Large Diagonalizing by Fixed--Points
}}
\end{center}

\vspace{1.5em}

\begin{abstract}
 A universal schema for diagonalization was popularized by {\sc N. S.
Yanofsky}  (2003) in which the
existence of a (diagonolized-out and contradictory) object implies
the existence of a fixed-point for a certain function. It was shown
that many self-referential paradoxes and diagonally proved theorems
can fit in that schema. Here, we fit more theorems in the universal
schema of diagonalization, such as {\sc Euclid}'s theorem on the
infinitude of the primes and new proofs of {\sc G. Boolos} (1997) for {\sc Cantor}'s theorem on the
non-equinumerosity of a set with its powerset.
Then, in Linear Temporal Logic, we show
the non-existence of a fixed-point in this logic whose
proof resembles the argument of
{\sc Yablo}'s paradox. Thus, Yablo's paradox turns for the
first time into a genuine  mathematico-logical theorem in
the framework of Linear Temporal Logic.
Again the diagonal schema of the paper is used in this proof;
and also it is shown that
{\sc G. Priest}'s
inclosure schema (1997) can fit in our
universal diagonal/fixed-point schema. We also show the
existence of dominating ({\sc Ackermann}-like) functions (which dominate a given countable set of functions---like
primitive recursives) using the
schema.

\bigskip

\bigskip

\centerline{${\backsim\!\backsim\!\backsim\!\backsim\!\backsim\!\backsim\!\backsim\!
\backsim\!\backsim\!\backsim\!\backsim\!\backsim\!\backsim\!\backsim\!
\backsim\!\backsim\!\backsim\!\backsim\!\backsim\!\backsim\!\backsim\!
\backsim\!\backsim\!\backsim\!\backsim\!\backsim\!\backsim\!\backsim\!
\backsim\!\backsim\!\backsim\!\backsim\!\backsim\!\backsim\!\backsim\!
\backsim\!\backsim\!\backsim\!\backsim\!\backsim\!\backsim\!\backsim\!
\backsim\!\backsim\!\backsim\!\backsim\!\backsim\!\backsim\!\backsim\!
\backsim\!\backsim\!\backsim\!\backsim\!\backsim\!\backsim\!\backsim\!
\backsim\!\backsim\!\backsim\!\backsim\!\backsim\!\backsim\!\backsim\!
\backsim\!\backsim\!\backsim\!\backsim\!\backsim\!\backsim\!\backsim\!
\backsim\!\backsim\!\backsim}$}

\bigskip


\end{abstract}

\bigskip

\medskip

\vfill

\hspace{.75em} \fbox{\textsl{\footnotesize Date: 30.08.14  (30 August 2014)}}

\vfill

\bigskip
\noindent\underline{\centerline{}}
\centerline{page 1 (of \pageref{LastPage})}


\newpage
\setcounter{page}{2}
\SetWatermarkAngle{65}
\SetWatermarkLightness{0.925}
\SetWatermarkScale{2.25}
\SetWatermarkText{\!\!\!\!\!\!\!\!\!\!\!\!\!\!\!\!\!\!\!\!\!\!
{\sc MANUSCRIPT (Submitted)}}



\section{Introduction}
{\sc Cantor}'s Diagonal Argument was introduced  in his
(third proof for the) famous theorem on  non--denumerability of the
reals; the argument shows that there can be no surjection from a
set $A$ to its powerset $\mathscr{P}(A)$: for any function
$F:A\rightarrow\mathscr{P}(A)$ the set ${D}_F=\{x\in A\mid
x\not\in F(x)\}$ is not  in the range of $F$ because for any $a\in
A$ we have $a\in {D}_F\longleftrightarrow a\not\in F(a)$, and
so $a\in({D}_F\setminus
F(a))\cup(F(a)\setminus {D}_F)$, whence ${D}_F\neq
F(a)$. This argument shows up also in {\sc Russell}'s Paradox: the
collection $R=\{x\mid x\not\in x\}$ of sets is not a set, since for
any set $A$ we have $A\in R\longleftrightarrow A\not\in A$, so
$A\neq R$.  One other example is {\sc Turing}'s Halting Problem in Computability Theory: if $W_0,
W_1, W_2, \cdots$ is the family of all {\sc re} sets (recursively
enumerable subsets of $\mathbb{N}$), then the set
$\overline{\mathcal{K}}=\{n\in\mathbb{N}\mid n\not\in W_n\}$ is not
{\sc re}, because for any {\sc re} set $W_m$ we have
$m\in\overline{\mathcal{K}}\longleftrightarrow m\not\in W_m$, and so
$m\in(\overline{\mathcal{K}}\setminus
W_m)\cup(W_m\setminus\overline{\mathcal{K}})$, thus
$\overline{\mathcal{K}}\neq W_m$. It can be seen that the (diagonal)
set $\mathcal{K}=\{n\in\mathbb{N}\mid n\in W_n\}$ is an {\sc re} but
undecidable set.

Many other theorems in mathematics (logic, set theory, computability
theory, complexity theory, etc.) use diagonal arguments;  {\sc Tarski}'s theorem on the undefinability of
truth, and {\sc G\"odel}'s theorem on the incompleteness of sufficiently
strong and ($\omega$--)consistent theories are two prominent examples. In 2003, {\sc Noson S.
Yanofsky} published the paper \cite{yan} mentioning some earlier
descriptions for ``many of the classical paradoxes and
incompleteness theorems in a categorial fashion", in the sense that
by using ``the language of category theory (and of cartesian closed
categories in particular)" one can demonstrate some paradoxical
phenomena and show the above mentioned theorems of {\sc Cantor, Tarski
and G\"odel}; the goal of \cite{yan} was ``to make these amazing
results available to a larger audience". In that paper, a universal
schema has been considered in the language of sets and functions
(not categories) and the paradoxes of the Liar, the strong liar,
{\sc Russell, Grelling, Richard}, Time Travel, and {\sc L\"ob}, and the theorems
of {\sc Cantor} ($A\lvertneqq\mathscr{P}(A)$), {\sc Turing} (undecidability of
the Halting problem, and existence of a non--{\sc re} set),
{\sc Baker-Gill-Solovay} (the existence of an oracle $\mathcal{O}$ such
that $\mathbf{P}^\mathcal{O}\neq\mathbf{NP}^\mathcal{O}$), {\sc Carnap}
(the diagonalization lemma), {\sc G\"odel} (first incompleteness theorem),
{\sc Rosser} (incompleteness of sufficiently strong and consistent
theories), {\sc Tarski} (undefinability of truth in sufficiently strong
languages), {\sc Parikh} (existence of sentences with very long proofs),
{\sc Kleene} (Recursion Theorem), {\sc Rice} (undecidability of non--trivial
properties of recursive functions), and {\sc von Neumann}  (existence of
self--reproducing machines) are shown as instances.

In this paper,
we fit some other theorems and proofs into the above mentioned
universal schema of {\sc Yanofsky};
these include {\sc Euclid}'s Theorem on the infinitude of the primes, {\sc Boolos}' proof of the existence of some explicitly definable counterexamples to the non--injectivity of  functions $F:\mathscr{P}(A)\rightarrow A$ for any set $A$, {\sc Yablo}'s paradox in a form of a mathematical theorem in the framework of linear temporal logic as a non--existence of some certain fixed--points, and the existence of dominating functions for a given countable set of functions like {\sc Ackermann}'s function which dominates all the primitive recursive functions.
In the rest of the introduction we
fix our notation and introduce the common framework.

\subsection{Cantor's Theorem by Fixed--Points}
Let $B$, $C$ and $D$ be arbitrary sets. Any function $f:B\times
C\rightarrow D$ corresponds to a function $\widehat{f}:C\rightarrow
D^B$ where $\widehat{f}(c)(b)=f(b,c)$ for any $b\in B$ and $c\in C$
(the set $D^B$ consists of all the functions from $B$ to $D$).
Conversely, for any function $F:C\rightarrow D^B$ there exists some
$f:B\times C\rightarrow D$ such that $\widehat{f}=F$: for any $b\in
B$ and $c\in C$ let $f(b,c)=F(c)(b)$. In the other words  \
$\widehat{ - }:D^{B\times C}\cong (D^B)^C$. Let $f:B\times
C\rightarrow D$ be a fixed function. A function $g:B\rightarrow D$
is called representable by $f$ at a fixed $c_0\in C$, when for any
$x\in B$, $g(x)=f(x,c_0)$ holds. In the other words,
$g=\widehat{f}(c_0)$. So, the function $\widehat{f}:C\rightarrow
D^B$ is onto if and only if every function $B\rightarrow D$ is
representable by $f$ at some $c_0\in C$.

\begin{theorem}[Cantor]\label{thm:cantor}
Assume the function $\alpha:D\longrightarrow D$, for a set $D$, does
not have any fixed point (i.e., $\alpha(d)\neq d$ for all $d\in D$).
Then for any set $B$ and any function $f:B\times B\rightarrow D$ there
exists a function $g:B\rightarrow D$ that is not representable by
$f$ (i.e., for all $b\in B$, $g(-)\neq f(-,b)$).
\end{theorem}

\begin{proof}
The desired function $g:x\mapsto\alpha(f(x,x))$ can be constructed
as follows:
\begin{diagram}
B\times B & \rTo^f & D \\
\uTo^{\triangle_B} & & \dTo_\alpha\\
B & \rTo_g & D\\
\end{diagram}
where $\triangle_B$ is the diagonal function of $B$
($\triangle_B(x)=\langle x,x\rangle$). If $g$ is representable by
$f$ at $b\in B$, then $g(x)=f(x,b)$ for any $x\in B$, and in
particular $g(b)=f(b,b)$. On the other hand by the definition of $g$
we have $g(x)=\alpha(f(x,x))$ and in particular (for $x=b$)
$g(b)=\alpha(f(b,b))$. It follows that $f(b,b)$ is a fixed--point of
$\alpha$; contradiction. Whence, the function $g$ is not
representable by $f$ (at any $b\in B$).
\end{proof}

For any set $A$ we have $\mathscr{P}(A)\cong\mathbf{2}^A$ where
$\mathbf{2}=\{0,1\}$ and $\mathbf{2}^A$ is the set of all functions
from $A$ to $\mathbf{2}$. So, Cantor's theorem is equivalent to the
non--existence of a surjection $A\rightarrow\mathbf{2}^A$. Putting
it another way, Cantor's theorem says that for any function
$f:A\times A\rightarrow\mathbf{2}$   there exists a function
$g:A\rightarrow\mathbf{2}$ which is not representable by $f$ (at any
member of $A$). In this new setting, Cantor's proof goes as follows:
let $\triangle_A:A\rightarrow A\times A$ be the diagonal function of
$A$ ($\triangle_A(x)=\langle x,x\rangle$) and let
$\alpha:\mathbf{2}\rightarrow\mathbf{2}$ be a fixed function. Define
$g:A\rightarrow\mathbf{2}$  by $g(x)=\alpha(f(\triangle_A(x)))$. If
$g$ is representable by $f$ and fixed $a\in A$ then
$f(a,a)=g(a)=\alpha(f(a,a))$, which shows that $\alpha$ has a
fixed--point (namely, $f(a,a)$). So, for reaching to a
contradiction, we need to take a function
$\alpha:\mathbf{2}\rightarrow\mathbf{2}$ which does not have any
fixed--point; and the only such function (without any fixed--point)
is the negation function ${\tt
neg}:\mathbf{2}\rightarrow\mathbf{2}$, ${\tt neg}(i)=1-i$ for
$i=0,1$. For a function $F:A\rightarrow\mathscr{P}(A)$ let
$f:A\times A\rightarrow\mathbf{2}$ be defined as
$$f(a,a')=\left\{\begin{array}{l} 1 \quad \textrm{if}\ \  a\in F(a')
\\ 0  \quad \textrm{if}\ \ a\not\in F(a')\\
\end{array} \right.$$
 The function $g$ constructed by the
diagram
\begin{diagram}
A\times A & \rTo^f & \mathbf{2} \\
\uTo^{\triangle_A} & & \dTo_{\tt neg}\\
A & \rTo_g & \mathbf{2}\\
\end{diagram}
is the characteristic function of the set $D_F=\{x\in A\mid x\not\in
F(x)\}$. That $g$ is not representable by $f$ (at any $a\in A$) is
equivalent to saying that the set $D_F$ is not in the range of $F$
(i.e., $D_F\neq F(a)$ for any $a\in A$).

In the rest of the paper  we will fit some theorems in the diagram
of the proof of Theorem~\ref{thm:cantor} by varying the sets $B$ and
$D$ (and the functions $f$ and $\alpha$); most of the times
$D=\mathbf{2}$ (and $\alpha={\tt neg}$) will hold, just like the
above diagram.

\section{Euclid's Theorem on the Infinitude of the Primes}
Our first instance of Cantor's Diagonal Proof starts with a
surprise: the ancient theorem of Euclid stating that there are
infinitely many prime numbers. We use (almost) the classical proof
of Euclid which seems far from being a diagonal argument. Indeed
there are many different proofs of this theorem in the literature,
and ours is not a new one; we just fit a proof in a diagonal diagram
as above.

\begin{theorem}[Euclid]\label{euclid}
There are infinitely many prime numbers in $\mathbb{N}$.
\end{theorem}

\begin{proof}
Define the function
$f:\mathbb{N}\times\mathbb{N}\rightarrow\mathbf{2}$     as follows:
$$f(n,m)=\left\{\begin{array}{l} 1 \quad  \textrm{if all the prime factors of}\  (n!+1)\  \textrm{are
less than}\  m
\\ 0  \quad \textrm{if some prime factor of}\ (n!+1) \ \textrm{is
greater than or equal to}\  m\\
\end{array} \right.$$
For example, $f(4,9)=1$
because $4!+1=25$ and it has no other prime factor but $5$ and
$5<9$; it can be seen that $f(4,m)=0$ for all $m\leqslant 5$ and
$f(4,m)=1$ for all $m>5$. Indeed, for any $n\in\mathbb{N}$ we have
$f(n,n)=0$ because no prime factor of $n!+1$ can be less than $n$:
for any $d<n$ if $d\mid (n!+1)$ then from $d\mid n!$ it follows that
$d\mid 1$ so $d$ cannot be a prime. Now, consider the function
$g:\mathbb{N}\rightarrow\mathbf{2}$ constructed as
\begin{diagram}
\mathbb{N}\times\mathbb{N} & \rTo^f & \mathbf{2} \\
\uTo^{\triangle_\mathbb{N}} & & \dTo_{\tt neg}\\
\mathbb{N} & \rTo_g & \mathbf{2}\\
\end{diagram}
If all the prime numbers are less than $\mathfrak{p}\in\mathbb{N}$
then the function $g$ is representable by $f$ at $\mathfrak{p}$: for
any $n\in\mathbb{N}$, $f(n,\mathfrak{p})=1$ and $g(n)={\tt
neg}(f(n,n))=1$; whence $g(n)=f(n,\mathfrak{p})$ for all
$n\in\mathbb{N}$. A contradiction follows as before: if such a
number $\mathfrak{p}$ exists, then $f(\mathfrak{p},\mathfrak{p})$
becomes a fixed--point of ${\tt neg}$. So, there exists no
$\mathfrak{p}\in\mathbb{N}$ such that all the primes are non--greater  than
$\mathfrak{p}$; whence there must be infinitely many prime numbers.
\end{proof}

\noindent
 This surprising argument, we believe, deserves another closer look:
define the function $F:\mathbb{N}\rightarrow\mathscr{P}(\mathbb{N})$ by $$F(n)=\{x\in\mathbb{N}\mid n \textrm{ is greater than or equal to all the prime factors of } (x!+1)\}.$$
Cantor's Theorem says that $F$ cannot be surjective, or, more explicitly, $D_F=\{n\mid n\not\in F(n)\}$ the (anti--diagonal) set  is not equal to any $F(m)$. A number--theoretic argument shows that $D_F=\mathbb{N}$ because for any $n$ all the prime factors of $(n!+1)$ are greater than $n$ (see the proof of the above Theorem~\ref{euclid}). On the other hand if $\mathfrak{p}\in\mathbb{N}$ is the greatest prime, then $F(\mathfrak{p})=\mathbb{N}=D_F$, a contradiction!

\section{Some Other Proofs for Cantor's Theorem}
In 1997 late George Boolos published another proof \cite{boolos} for
Cantor's Theorem, by showing that there cannot be any injection from
the powerset of a set to the set. This proof has been (implicitly or
explicitly) mentioned also in \cite{gch,yet} (but without referring
to the earlier \cite{boolos}). The first proof is
essentially Cantor's Diagonal Argument.

\begin{theorem}
No function $h:\mathscr{P}(A)\rightarrow A$ can be injective.
\end{theorem}

\begin{proof}
Let $h:\mathscr{P}(A)\rightarrow A$ be a function. Define
$f:\mathscr{P}(A)\times\mathscr{P}(A)\rightarrow\mathbf{2}$ by
$$f(X,Y)=\left\{\begin{array}{l} 1 \quad  \textrm{if}\ \ h(X)\not\in Y\\
\\ 0  \quad \textrm{if}\ \ h(X)\in Y \\ \end{array} \right.$$
 and let
$g:\mathscr{P}(A)\rightarrow\mathbf{2}$ be the following function
\begin{diagram}
\mathscr{P}(A)\times\mathscr{P}(A) & \rTo^f & \mathbf{2} \\
\uTo^{\triangle_{\mathscr{P}(A)}} & & \dTo_{\tt neg}\\
\mathscr{P}(A) & \rTo_g & \mathbf{2}\\
\end{diagram}
Let $\mathcal{D}_h=\{a\in A\mid \exists Y\subseteq A \ (h(Y)=a \ \&
\ a\not\in Y)\}$. Note that for any $X\subseteq A$ we have
$h(X)\not\in X\longrightarrow h(X)\in\mathcal{D}_h$. We show that if
$h$ is one--to--one then $g$ is representable by $f$ at
$\mathcal{D}_h$. For, if $h$ is injective then for any $X\subseteq
A$,

\begin{tabular}{rcl}
$h(X)\!\in\!\mathcal{D}_h$ & $\longrightarrow$ & $\exists
Y\!\subseteq\!A\;
(h(Y)\!=\!h(X)\ \& \ h(X)\not\in Y)$ \\
  & $\longrightarrow$ & $\exists Y\; (Y\!=\!X \ \& \ h(X)\not\in Y)$ \\
  & $\longrightarrow$ & $h(X)\not\in X$ \\
\end{tabular}

\noindent  Whence, $h(X)\not\in X\longleftrightarrow
h(X)\in\mathcal{D}_h$ for all $X\subseteq A$. So, for any
$X\subseteq A$,

\begin{tabular}{rcl}
$f(X,\mathcal{D}_h)=0$ & $\longleftrightarrow$ & $h(X)\in\mathcal{D}_h$ \\
  & $\longleftrightarrow$ & $h(X)\not\in X$ \\
  & $\longleftrightarrow$ & $f(X,X)=1$ \\
  & $\longleftrightarrow$ & $g(X)={\tt neg}(f(X,X))=0$ \\
\end{tabular}

\noindent
 Thus, $g(X)=f(X,\mathcal{D}_h)$. The contradiction
(that ${\tt neg}$ possesses a fixed--point) follows as before,
implying that the function $h$ cannot be injective.
\end{proof}

In fact the proof of the above theorem  gives some more information
than mere non--injectivity of any function
$h:\mathscr{P}(A)\rightarrow A$, i.e., the existence of some
$C,D\subseteq A$ such that $h(C)=h(D)$ and $C\neq D$.

\begin{corollary}\label{cor:boolos}
For any function $h:\mathscr{P}(A)\rightarrow A$ there are some
$C,D\subseteq A$ such that $h(C)=h(D)\in D\setminus C$ (and so
$C\neq D$).
\end{corollary}

\begin{proof}
For any $X\subseteq A$ we had $h(X)\not\in X\longrightarrow
h(X)\in\mathcal{D}_h$, whence $h(\mathcal{D}_h)\not\in
\mathcal{D}_h\longrightarrow h(\mathcal{D}_h)\in\mathcal{D}_h$, and
so $h(\mathcal{D}_h)\in\mathcal{D}_h$. Thus, there exists some
$\mathcal{C}_h$ such that $h(\mathcal{C}_h)=h(\mathcal{D}_h)$ and
$h(\mathcal{D}_h)\not\in\mathcal{C}_h$. So, for these
$\mathcal{C}_h,\mathcal{D}_h\subseteq A$ we have
$h(\mathcal{C}_h)=h(\mathcal{D}_h)\in\mathcal{D}_h\setminus\mathcal{C}_h$.
\end{proof}

Boolos notes in \cite{boolos} that, in the above proof, though the
set $\mathcal{D}_h$ had an explicit definition:
\newline\centerline{$\mathcal{D}_h=\{a\in A\mid \exists Y\subseteq A \ (h(Y)=a \ \& \
a\not\in Y)\}$,} the set $\mathcal{C}_h$ was not defined explicitly,
and its mere existence was shown. So, this proof of non--injectivity
was not constructive (did not explicitly construct two sets $C$ and
$D$ such that $h(C)=h(D)$ and $C\neq D$). For  a constructive
proof, Boolos \cite{boolos} proceeds as follows (cf.
\cite{gch,yet}).

Fix a function $h:\mathscr{P}(A)\rightarrow A$. Call a subset
$B\subseteq A$ an {\em $h$--woset} ($h$ well ordered set) when there
exists a well ordering  $\prec$ on $B$ such that $b=h(\{x\in B\mid
x\prec b\})$ for any $b\in B$.
 For example, $\{h(\emptyset)\}$ is an $h$--woset, and indeed any
 non--empty $h$--woset must contain $h(\emptyset)$. Some other
 examples of $h$--wosets are 

 $\{h(\emptyset), h\big(\big{\{}h(\emptyset)\big{\}}\big)\}$ and
  $\{h(\emptyset), h\big(\big{\{}h(\emptyset)\big{\}}\big),
  h\Big(\Big{\{}h(\emptyset), h\big(\big{\{}h(\emptyset)\big{\}}\big)
 \Big{\}}\Big)\}$, etc.

\noindent We need the following two facts about the $h$--wosets:

\begin{itemize}\parskip=-5pt
\item[(1)] If $B$ and $C$ are two $h$--wosets with the well
ordering relations $\prec_B$ and $\prec_C$ then exactly one (and only one) of the following holds:

\begin{itemize}\parskip=-2pt
\item[(i)] $(B,\prec_B)$ is an initial segment of $(C,\prec_C)$, or

\item[(ii)] $(C,\prec_C)$ is an initial segment of $(B,\prec_B)$, or

\item[(iii)] $(B,\prec_B)=(C,\prec_C)$.
\end{itemize}

 \item[(2)] For any $h$--woset $B$, if $h(B)\not\in B$ then the
set $\Phi(B)=B\cup\{h(B)\}$ is an $h$--woset, and $B$ is an initial
segment of $\Phi(B)$.
\end{itemize}

The statement (1) corresponds to Zermelo's theorem that any two well
ordered sets are comparable to each other: either they are
isomorphic or one of them is isomorphic to an initial segment of the
other one. It follows from (1) that the union of all $h$--wosets is
an $h$--woset, denoted by $\mathcal{W}_h$; thus $\mathcal{W}_h$ is
the greatest $h$--woset. For (2) let $B$ be an $h$--woset with the
well ordering $\prec_B$ such that $h(B)\not\in B$. Then $\Phi(B)$ is
an $h$--woset with the well ordering $\prec_{\Phi(B)}=\prec_B\cup\
(B\times\{h(B)\})$.

The proof of Boolos \cite{boolos} continues as follows (see also
\cite{gch}): since $\Phi(\mathcal{W}_h)=\mathcal{W}_h$ then
$h(\mathcal{W}_h)\in \mathcal{W}_h$. Also for $\mathcal{V}_h=\{x\in
\mathcal{W}_h\mid x\prec_{\mathcal{W}_h}h(\mathcal{W}_h)\}$ we have
$h(\mathcal{W}_h)=h(\mathcal{V}_h)$ and
$\mathcal{W}_h\neq\mathcal{V}_h$ because $h(\mathcal{W}_h)\not\in
\mathcal{V}_h$. Indeed, the result is stronger than this (and
Corollary~\ref{cor:boolos}) since the sets $\mathcal{W}_h$ and
$\mathcal{V}_h$ were explicitly defined in a way that
$\mathcal{V}_h\subsetneqq\mathcal{W}_h$ holds and
$h(\mathcal{V}_h)=h(\mathcal{W}_h)\in\mathcal{W}_h\setminus\mathcal{V}_h$.
 As another partial surprise we show that this proof is also
 diagonal and fits in our universal framework.

\begin{theorem}[Boolos]\label{thm:boolos}
For any set $A$ and function $h:\mathscr{P}(A)\rightarrow A$ there
exist explicitly definable subsets $V,W\subseteq A$ such that
$V\subsetneqq W$ and $h(V)=h(W)\in W\setminus V$.
\end{theorem}

\begin{proof}
Let $\mathbf{W}_h$ be the class of all $h$--wosets; i.e., all
subsets $B\subseteq A$ on which there exists a (unique) well
ordering $\prec_B$ such that $b=h(\{x\in B\mid x\prec_B b\})$ for
all $b\in B$. Define $\Phi:\mathbf{W}_h\rightarrow\mathbf{W}_h$ by
$$\Phi(X)=\left\{\begin{array}{ll} X\cup\{h(X)\} & \quad \textrm{if}\ \  h(X)\not\in X
\\ X  & \quad \textrm{if}\ \ h(X)\in X\\
\end{array} \right.$$
 $$\textrm{with}\quad\  \prec_{\Phi(X)}=\left\{\begin{array}{ll} \prec_X\cup\ (X\times\{h(X)\}) & \quad \textrm{if}\ \  h(X)\not\in X
\\ \prec_X  & \quad \textrm{if}\ \ h(X)\in X\\
\end{array} \right.$$ Define the function
$f:\mathbf{W}_h\times\mathbf{W}_h\rightarrow\mathbf{2}$    by
$$f(X,Y)=\left\{\begin{array}{lll} 1  &   \ \ \textrm{if}\   \Phi(X) \ \textrm{is isomorphic
to}\ Y\ \textrm{or an initial segment of it} & \big(\Phi(X)\sqsubseteq Y\big)
\\ 0  &   \ \ \textrm{if}\  Y \ \textrm{is isomorphic to an initial segment of}\ \Phi(X) &  \big(Y\sqsubset\Phi(X)\big)\\
\end{array} \right.$$
Let
$\mathcal{W}_h$ be the greatest element of $\mathbf{W}_h$ (as
above). Then $f(X,\mathcal{W}_h)=1$ for all $X\in\mathbf{W}_h$. We
claim that
\newline\centerline{$(\ast)$\quad there exists some $Z\in\mathbf{W}_h$ such that $h(Z)\in
Z$ or equivalently $\Phi(Z)=Z$.} Assume (for a moment) that the
claim is false. Then for all $X\in\mathbf{W}_h$, $X$ is (isomorphic
to) an initial segment of $\Phi(X)$; whence $f(X,X)=0$. Let
$g:\mathbf{W}_h\rightarrow\mathbf{2}$ be   as
\begin{diagram}
\mathbf{W}_h\times\mathbf{W}_h & \rTo^f & \mathbf{2} \\
\uTo^{\triangle_{\mathbf{W}_h}} & & \dTo_{\tt neg}\\
\mathbf{W}_h & \rTo_g & \mathbf{2}\\
\end{diagram}
It follows from assuming the falsity of the claim $(\ast)$ that  $$g(X)={\tt
neg}(f(X,X))=1=f(X,\mathcal{W}_h).$$ Thus $g$ is representable by $f$
(at $\mathcal{W}_h$) and the usual contradiction (the existence of a
fixed--point for ${\tt neg}$) follows. So, the claim $(\ast)$ is
true, and there exists some $Z\in\mathbf{W}_h$ such that $h(Z)\in
Z$ or equivalently $\Phi(Z)=Z$. It can be seen that then
$\mathcal{W}_h=Z$, so  $\Phi(\mathcal{W}_h)=\mathcal{W}_h$ and
$h(\mathcal{W}_h)\in\mathcal{W}_h$. Whence, as above, for the subset
$\mathcal{V}_h=\{x\in \mathcal{W}_h\mid
x\prec_{\mathcal{W}_h}h(\mathcal{W}_h)\} \ (\subseteq A)$ we will
have $\mathcal{V}_h\subsetneqq\mathcal{W}_h$ and
$h(\mathcal{V}_h)=h(\mathcal{W}_h)\in\mathcal{W}_h\setminus\mathcal{V}_h$.
Note that both $\mathcal{W}_h$ and $\mathcal{V}_h$ were defined
explicitly.
\end{proof}

Let us reiterate what was proved:

\bigskip

\noindent (Corollary~\ref{cor:boolos}) For any function
$h:\mathscr{P}(A)\rightarrow A$ a subset $\mathcal{D}_h\subseteq A$
was explicitly defined in such a way that there exists some
$\mathcal{C}_h\subseteq A$ (without an explicit definition) such
that $\mathcal{C}_h\neq\mathcal{D}_h$ and
$h(\mathcal{C}_h)=h(\mathcal{D}_h)\in\mathcal{D}_h\setminus\mathcal{C}_h$.

\bigskip

\noindent (Theorem~\ref{thm:boolos}) For any function
$h:\mathscr{P}(A)\rightarrow A$ two subset $\mathcal{V}_h\subseteq
A$ and $\mathcal{W}_h\subseteq A$ were explicitly defined in such a
way that $\mathcal{V}_h\subsetneqq\mathcal{W}_h$ and
$h(\mathcal{V}_h)=h(\mathcal{W}_h)\in\mathcal{W}_h\setminus\mathcal{V}_h$.

\section{Yablo's Paradox}
To counter a general belief that all the paradoxes stem from a kind
of circularity (or involve some self--reference, or use a diagonal
argument) Stephen Yablo designed a paradox in 1985 that seemingly
avoided self--reference (\cite{yab1,yab2}).
Let us  fix  our reading of Yablo's
Paradox. Consider the sequence of sentences
$\{\mathcal{Y}_n\}_{n\in\mathbb{N}}$ such that for each
$n\in\mathbb{N}$:
\newline\centerline{$\mathcal{Y}_n \textrm{ is true } \Longleftrightarrow
\forall k>n\ (\mathcal{Y}_k$ is untrue).}

\noindent The paradox follows from
the following deductions. For each $n\in\mathbb{N}$,

\begin{tabular}{ccl}
\qquad $\mathcal{Y}_n$ is true & $\Longrightarrow$ &  $\forall k>n\ (\mathcal{Y}_k$ is untrue)\\
\qquad  & $\Longrightarrow$ & $(\mathcal{Y}_{n+1}$ is untrue)  and
 $\forall k>n+1\ (\mathcal{Y}_k$ is untrue) \\
\qquad  & $\Longrightarrow$ & $(\mathcal{Y}_{n+1}$ is untrue)  and
 $(\mathcal{Y}_{n+1}$ is true),  \\
 \qquad  & $\Longrightarrow$ & Contradiction! \\
\end{tabular}

\noindent Thus $\mathcal{Y}_n$ is not true. So,
\newline\centerline{$\forall k\ (\mathcal{Y}_k$ is untrue),}

\noindent and in particular \newline\centerline{$\forall k>0\ (\mathcal{Y}_k$
is untrue),}

\noindent and so $\mathcal{Y}_0$ must be true (and untrue at
the same time); contradiction!

\subsection{Propositional Linear Temporal Logic}
The propositional
linear temporal logic (LTL) is a logical formalism that can refer to
time; in LTL  one can encode formulae about the future, e.g., a
condition will eventually be true, a condition will be true until
another fact becomes true, etc. LTL was first proposed for the
formal verification of computer programs in 1977 by Amir Pnueli~\cite{pnueli}. For a modern introduction to LTL and its syntax and
semantics see e.g. \cite{temporal}. Two modality operators in LTL
that we will use are the ``next" modality $\Circle$ and the ``always"
modality $\Box$.  The formula $\Circle\varphi$ holds (in the
current moment) when $\varphi$ is true in the ``next step", and the
formula $\Box\varphi$ is true (in the current moment) when
$\varphi$ is true ``now and forever" (``always in the future"). In
the other words, $\Box$ is the reflexive and transitive closure
of $\Circle$.
It can be seen that  the formula $\Circle\neg\varphi\longleftrightarrow\neg\Circle\varphi$ is always true (is a law of LTL, see T1 on page 27 of \cite{temporal}), since $\varphi$ is untrue in the next step if and only if it is not the case that ``$\varphi$ is true in the next step". Also the formula $\Circle\Box\psi$ is true when
$\psi$ is true from the next step onward, that is $\psi$ holds in
the next step, and the step after that, and the step after that, etc.
The same holds for $\Box\Circle\psi$; indeed the formula
$\Circle\Box\psi\longleftrightarrow\Box\Circle\psi$ is a
law of LTL (T12 on page 28  of \cite{temporal}).
Whence, we have the equivalences $\Circle\Box\neg\varphi\longleftrightarrow\Box\Circle\neg\varphi\longleftrightarrow\Box\neg\Circle\varphi$
in LTL.

 \noindent The intended models (semantics) of LTL are systems $\langle\mathbb{N},\Vdash\rangle$ where $\Vdash\ \subseteq \mathbb{N}\!\times\!{\tt Atoms}$ is an arbitrary relation which can be extended to all formulas as follows:

\medskip
 \begin{tabular}{lll}
 $n\Vdash \varphi\wedge\psi$ & if and only if & $n\Vdash\varphi$ and $n\Vdash\psi$, \\
 $n\Vdash\neg\varphi$ & if and only if & $n\not\Vdash\varphi$, \\
 $n\Vdash\Circle\varphi$ & if and only if & $(n+1)\Vdash\varphi$, \\
 $n\Vdash\Box\varphi$ & if and only if & $m\Vdash\varphi$ for every $m\geq n$. \\
 \end{tabular}
\medskip

\noindent A formula $\tau$ is called valid (an LTL tautology) when for any model $\langle\mathbb{N},\Vdash\rangle$ and any $n\in\mathbb{N}$ we have $n\Vdash\tau$. Here we can readily check the validity of the formula $\Circle\neg\varphi\longleftrightarrow\neg\Circle\varphi$ as follows:
$$n\Vdash\Circle\neg\varphi\Longleftrightarrow(n+1)\Vdash\neg\varphi\Longleftrightarrow
(n+1)\not\Vdash\varphi\Longleftrightarrow n\not\Vdash\Circle\varphi \Longleftrightarrow n\Vdash\neg\Circle\varphi.$$
Also the validity of $\Circle\Box\psi\longleftrightarrow\Box\Circle\psi$ can be readily checked:

\begin{tabular}{lll}
$n\Vdash\Circle\Box\varphi$ & $\Longleftrightarrow$ & $(n+1)\Vdash\Box\varphi$ \\
 & $\Longleftrightarrow$ & $\forall k\geq n+1\big(k\Vdash\varphi\big)$ \\
 & $\Longleftrightarrow$ & $\forall k\geq n \big[(k+1)\Vdash\varphi\big]$ \\
 & $\Longleftrightarrow$ & $\forall k\geq n\big(k\Vdash\Circle\varphi\big)$ \\
 &  $\Longleftrightarrow$ & $n\Vdash\Box\Circle\varphi$. \\
\end{tabular}

\noindent Now we show the non--existence of a formula
 $\mathscr{Y}$ that satisfies the equivalence
$$\mathscr{Y}\!\longleftrightarrow\!\Circle\Box\neg
\mathscr{Y} \ \ \big(\!\!\longleftrightarrow\!\Box\Circle\neg\mathscr{Y}
\longleftrightarrow\Box\neg\Circle\mathscr{Y}\big);$$
 in  other words $\mathscr{Y}$ is a fixed--point of the
operator $x\mapsto\Circle\Box\neg x\ \big(\!\!\equiv\Box\Circle\neg x\equiv\Box\neg\Circle x\big)$. Following \cite{yan} we can
demonstrate this by the following diagram
\begin{diagram}
{\sf LTL}\times {\sf LTL} & \rTo^f & \mathbf{2} \\
\uTo^{\triangle_{\sf LTL}} & & \dTo_{\tt neg}\\
{\sf LTL} & \rTo_g & \mathbf{2}\\
\end{diagram}
where {\sf LTL} is the set of sentences in the language of LTL and
$f$ is defined by
$$f(X,Y)=\left\{\begin{array}{l} 1 \quad \textrm{if}\ \  X\not\equiv\Circle\Box\neg Y
\\ 0  \quad \textrm{if}\ \ X\equiv\Circle\Box\neg Y\\
\end{array} \right.$$
Here, $g$
is the characteristic function of all the Yablo--like sentences, the
sentences which claim that all they say in the future (from the next
step onward) is untrue.

\begin{theorem}\label{ltlyablo}
For any   $\varphi$, the formula $\big(\varphi\leftrightarrow\Circle\Box\neg\varphi\big)$ is not provable in LTL.
\end{theorem}
\begin{proof}
If LTL proves  $\psi\leftrightarrow\Circle\Box\neg\psi$ for some (propositional)  formula $\psi$, then for a model $\langle\mathbb{N},\Vdash\rangle$:
\begin{itemize}\parskip=-2pt
\item[(i)] If $m\Vdash\psi$ for some $m$, then $m\Vdash\Circle\Box\neg\psi$ so $(m+1)\Vdash\Box\neg\psi$, hence $(m+i)\Vdash\neg\psi$ for all $i\geq 1$.
        In particular, $(m+1)\Vdash\neg\psi$ and $(m+j)\Vdash\neg\psi$ for all $j\geq 2$ which implies $(m+2)\Vdash\Box\neg\psi$ or $(m+1)\Vdash\Circle\Box\neg\psi$ so $(m+1)\Vdash\psi$, a contradiction!
\item[(ii)] So for all $k$ we have  $k\Vdash\neg\psi$  or equivalently $k\Vdash\neg\Circle\Box\neg\psi$ or $k\Vdash\Circle\neg\Box\neg\psi$, thus $(k+1)\Vdash\neg\Box\neg\psi$; hence $(k+n)\Vdash\psi$ for some $n\geq 1$, contradicting  (i)!
\end{itemize}
So, LTL$\not\vdash\big(\varphi\leftrightarrow\Circle\Box\neg\varphi\big)$ for all formulas $\varphi$.
\end{proof}

The above proof is very similar to Yablo's argument (in his paradox) presented at the beginning of this section, and this goes to say that Yablo's paradox has turned into a genuine mathematico--logical theorem (in LTL) for the first time in Theorem~\ref{ltlyablo}\footnote{Note that Yablo's paradox has already been used to give new proofs of some old theorems e.g. in \cite{cieslinski2} (for G\"odel's theorem) or in \cite{leach} (for Rosser's Theorem); but no new theorem had come out of Yablo's paradox.}.

\subsection{Priest's Inclosure Schema}
In 1997 Priest \cite{priest}   has shown the existence of a formula $Y(x)$
which satisfies
$Y(n)\leftrightarrow\forall
k\!>\!n\ \neg\mathcal{T}(\ulcorner Y(k)\urcorner)$ for every
$n\in\mathbb{N}$, where $\mathcal{T}(x)$ is a (supposedly truth)
predicate; here $\ulcorner\psi\urcorner$ is the (G\"odel) code of
the formula $\psi$.
 Here
 we construct a formula $Y(x)$ which, for every $n\in\mathbb{N}$, satisfies the formula
$Y(n)\leftrightarrow\forall k\!>\!n\ \Psi(\ulcorner
Y({k})\urcorner)$ for some $\Pi_1$ formula $\Psi$, by
using the Recursion Theorem (of Kleene); for recursion--theoretic
definitions and theorems see e.g. \cite{epstein}\footnote{Of course the mere existence of such a formula $Y(x)$ follows directly from G\"odel's Diagonal Lemma (\cite{epstein})}. Let $\mathbf{T}$
denote Kleene's T Predicate, and for a fixed $\Pi_1$ formula
$\Psi(x)$ let $r$ be the recursive function defined by $r(x,y)=\mu\,
z \big(z\!>\!x \ \&\ \neg\Psi(\ulcorner\neg\exists
u\mathbf{T}(y,{z},u)\urcorner)\big)$; note that $\neg\Psi$
is a $\Sigma_1$ formula. By the S--m--n theorem there exists a
primitive recursive function $s$ such that
$\varphi_{s(y)}(x)=r(x,y)$; here $\varphi_n$ denotes the unary
recursive function with (G\"odel) code $n$, so
$\varphi_0,\varphi_1,\varphi_2,\cdots$ lists all the unary recursive
functions. By Kleene's Recursion Theorem, there exists some (G\"odel
code) $e$ such that $\varphi_e=\varphi_{s(e)}$. Whence,
$$\varphi_{e}(x)=\varphi_{s(e)}(x)=r(x,e)=\mu\, z\big(z\!>\!x\ \& \
\neg\Psi(\ulcorner\neg\exists
u\mathbf{T}(e,{z},u)\urcorner)\big).$$ So, for any
$x\in\mathbb{N}$ we have $ \exists
u\mathbf{T}(e,{x},u)\Leftrightarrow
\varphi_e(x)\!\downarrow\ \Leftrightarrow  \exists z \big(z\!>\!x\
\& \ \neg\Psi(\ulcorner\neg\exists
u\mathbf{T}(e,{z},u)\urcorner)\big)$, or in the other
words  we have the equivalence $$\neg\exists
u\mathbf{T}(e,{x},u) \iff \forall z\!>\!x\
\Psi(\ulcorner\neg\exists
u\mathbf{T}(e,{z},u)\urcorner).$$ Thus if we let
$\mathcal{Y}(v) = \neg\exists z\mathbf{T}(e,\underline{v},z)$, then
for any $n\in\mathbb{N}$ we have
$$\mathcal{Y}(n)\leftrightarrow\forall k\!>\!n\ \Psi(\ulcorner
\mathcal{Y}({k})\urcorner).$$ Let us note that Yablo's
paradox occurs when $\Psi$ is taken to be an untruth (or
non-satisfaction) predicate; in fact one might be tempted to take
$\neg{\it Sat}_{\Pi,1}(x,\emptyset)$ (see Theorem~1.75 of \cite{hp})
as $\Psi(x)$; but by construction ${\it Sat}_{\Pi,1}(x,\emptyset)$ is $\Pi_1$ and so
$\neg{\it Sat}_{\Pi,1}(x,\emptyset)$ is $\Sigma_1$, and our proof
works for  $\Psi\in\Pi_1$ only (otherwise the function $r$ could not
be recursive). Actually, the above construction shows that the predicate ${\it
Sat}_{\Pi,1}(x,\emptyset)$ (in \cite{hp}) cannot be $\Sigma_1$,
which is equivalent to saying that the set of (arithmetical) true $\Pi_1$ sentences cannot
be recursively enumerable, and this is a consequence of G\"odel's
first incompleteness theorem\footnote{This line of reasoning also shows the non--existence of
a formula $\theta(x)$ (in arithmetical languages) which can satisfy the equivalence $\theta(x)\leftrightarrow\forall y\!>\!x\, \neg
\theta(y)$  in $\mathbb{N}$ or in a  theory containing Peano's Arithmetic.}.

In \cite{priest} Priest also introduced his Inclosure Schema and
showed that Yablo's paradox is amenable in it (see also \cite{bc2}). In the following,  we show that Priest's Inclosure
Schema can fit in Yanofsky's framework \cite{yan}.
With some inessential modification for better reading, Priest's
inclosure schema is defined to be a triple
$\langle\Omega,\Theta,\delta\rangle$ where
\begin{itemize}\parskip=-5pt
\item $\Omega$ is a set of objects;
\item $\Theta\subseteq\mathscr{P}(\Omega)$ is a property of subsets
of $\Omega$ such that $\Omega\in\Theta$;
\item $\delta:\Theta\rightarrow\Omega$ is a function such that for each $X\in\Theta$,
$\delta(X)\not\in X$.
\end{itemize}
That any inclosure schema is contradictory can be seen from the fact
that by the second item  $\delta(\Omega)$ must be defined and belong
to $\Omega$, but at the same time by the third item
$\delta(\Omega)\not\in\Omega$. We show how this can be proved by the
non--existence of a fixed--point for the negation function.

\begin{theorem}
If an inclosure schema exists, then  ${\tt negation}$  has a
fixed--point.
\end{theorem}

\begin{proof}
Assume $\langle\Omega,\Theta,\delta\rangle$ is a
(hypothetical) inclosure schema. Put
$f:\Theta\times\Theta\rightarrow\mathbf{2}$ as
$$f(X,Y)=\left\{\begin{array}{l} 1 \quad \textrm{if}\ \  \delta(X)\in Y
\\ 0  \quad \textrm{if}\ \ \delta(X)\not\in Y\\
\end{array} \right.$$
and let
$g:\Theta\rightarrow\mathbf{2}$ be defined as
\begin{diagram}
\Theta\times\Theta & \rTo^f & \mathbf{2} \\
\uTo^{\triangle_{\Theta}} & & \dTo_{\tt neg}\\
\Theta & \rTo_g & \mathbf{2}\\
\end{diagram}
We show that $g$ is representable by $f$ at $\Omega$. For every $X\!\in\!\Theta$ we have
$f(X,\Omega)=1$. On the other hand by the property of $\delta$, for
any $X\!\in\!\Theta$, $\delta(X)\!\not\in\!X$, and so $f(X,X)=0$,
thus $g(X)={\tt neg}(f(X,X))=1$. Whence $g(X)=f(X,\Omega)$ for all
$X\!\in\!\Theta$.
\end{proof}

\section{Dominating  Functions}
Ackermann's function is a recursive (computable) function which is
not primitive recursive (see e.g. \cite{epstein}). The class of
primitive recursive functions is the smallest class which contains
the initial functions, i.e.,
\begin{itemize}\parskip=-5pt
\item the constant zero function ${\tt
z}(x)=0$,
\item the successor function ${\tt s}(x)=x+1$ and
\item  the
projection functions ${\tt p}_i^n(x_1,\cdots,x_n)=x_i$ for any
$1\leqslant i\leqslant n\in\mathbb{N}$,
\end{itemize}
   and is closed under
   \begin{itemize}\parskip=-5pt
   \item composition and
   \item primitive recursion,
   \end{itemize}
 i.e., for primitive recursive functions $f,f_1,\cdots,f_n$ the function
 ${\tt comp}(f;f_1,\ldots,f_n)$ defined by $(x_1,\cdots,x_m)\mapsto f(f_1(x_1,\cdots,x_m),\ldots,f_n(x_1,\cdots,x_m))$
 is also
 primitive recursive,  and also  for primitive recursive functions
  $g$ and $h$ the function ${\tt prim.rec}(g,h)$ defined by $(x_1,\cdots,x_n,0)\mapsto g(x_1,\cdots,x_n)$
  and \newline\centerline{\quad $(x_1,\cdots,x_n,x+1)\mapsto h\big({\tt prim.rec}(g,h)(x_1,\cdots,x_n,x),x_1,\cdots,x_n,x\big)$} is also primitive
  recursive.

  The class of recursive functions contains the same initial functions and is closed under composition, primitive recursion, and also
  \begin{itemize}\parskip=-5pt
\item  minimization,
  \end{itemize}
  i.e., for recursive function $f$ the function
   ${\tt min}(f)$ defined by $(x_1,\cdots,x_n)\mapsto y$ where $y$ is
   the least natural number that satisfies  $f(x_1,\cdots,x_n,y)=0$ is also recursive;
   note that then  for all $z<y$ we have $f(x_1,\cdots,x_n,z)\neq 0$,
    and if there is no such $y$ then ${\tt min}(f)$ is undefined
    on $x_1,\cdots,x_n$.

    In fact, Ackermann's function is not only a
  non--primitive recursive (and a recursive) function, but it also
  dominates all the primitive recursive functions (see e.g. \cite{epstein}). A function
  $g$ is said to dominate a function $f$ (or $f$ is dominated by
  $g$) when for all but finitely many $x$'s  the inequality
  $g(x)>f(x)$ holds. Here we show a way of dominating a given
  enumerable list of functions by diagonalization. Before that let
  us note that the set of all primitive recursive functions can be
  (recursively) enumerated:  let $\#(f)$ denote the (G\"odel)
  code of the function $f$ and define  the G\"odel code of a
  primitive recursive function inductively:

\noindent \begin{tabular}{llr}
 $\bullet\,\,\,\#({\tt z})=1$, & $\bullet\,\,\,\#({\tt s})=2$,  & $\bullet\,\,\,\#({\tt p}_i^n)=2^i\cdot 3^n$, \\
  $\bullet\,\,\,\#({\tt comp}(f;f_1,\ldots,f_n))=5^{\#(f)}\cdot 7^{\#(f_1)}
 \cdots\mathfrak{p}_{n+2}^{\#(f_n)}$,  & and &  $\bullet\,\,\,\#({\tt prim.rec}(g,h))=3^{\#(g)}\cdot 5^{\#(h)}$,     \end{tabular}

\noindent
   where $\mathfrak{p}_i$ is
 the $i-$th prime number (thus, $\mathfrak{p}_0=2, \mathfrak{p}_1=3,
 \mathfrak{p}_2=5, \mathfrak{p}_3=7,  \cdots$). Let $\nu_n$ be
 the primitive recursive function with code $n$, if $n$ is a code of such a
 function; if $n$ is not a code for a primitive recursive function
 (such as $n=3$ or $n=10$) then let $\nu_n$ be the constant zero
 function ${\tt z}$.
  So, $\nu_0$, $\nu_1$, $\nu_2$, $\cdots$ lists all the
  primitive recursive functions.
 We show the existence of a unary function  that dominates all the functions $\nu_i$'s
  in the above list.

\begin{theorem}\label{ack}
For a list of  functions
$f_1,f_2,f_3,\cdots:\mathbb{N}\rightarrow\mathbb{N}$,
 there exists a unary function  $\mathbb{N}\rightarrow\mathbb{N}$ that dominates them all.
\end{theorem}

\begin{proof}
Define the   function
$f:\mathbb{N}\times\mathbb{N}\rightarrow\mathbb{N}$  as
$f(n,m)=\max_{(i\leqslant n)}f_i(m)$ and let $g$ be  defined by the following diagram where ${\tt s}$ is the successor function:
\begin{diagram}
\mathbb{N}\times\mathbb{N} & \rTo^f & \mathbb{N} \\
\uTo^{\triangle_{\mathbb{N}}} & & \dTo_{\tt s}\\
\mathbb{N} & \rTo_g & \mathbb{N}\\
\end{diagram}
  In fact, the function
$g:\mathbb{N}\rightarrow\mathbb{N}$ is defined as
$g(x)=\max_{(i\leqslant x)}f_i(x)+1$. Since the successor function
does not have any fixed--point, the function $g$ is not equal to any
of $f_i$'s. Moreover,   $g$ dominates all the $f_i$'s, since for any
$m\in\mathbb{N}$ and any $x\!\geqslant\!m$ by the definition of $g$
we have {$g(x)\!>\!\max_{(i\leqslant x)}f_i(x)\!\geqslant\!f_m(x)$.}
\end{proof}

For dominating the primitive recursive functions (some of which are
not unary) we can consider their unarized version: let
$\rho_0,\rho_1,\rho_2,\cdots$ be the list of unary functions
$\mathbb{N}\rightarrow\mathbb{N}$ defined as
$\rho_i(x)=\nu_i(x,\cdots,x)$. Whence
$\rho_0,\rho_1,\rho_2,\cdots$ lists all the unary primitive
recursive functions, and the construction of Theorem~\ref{ack}
produces a unary function which dominates all the unary primitive
recursive functions. Let us note that the function $g$ obtained in
the proof of Theorem~\ref{ack} is computable (intuitively) and so
recursive (by Church's Thesis); one can show directly that the above
function $g$ is recursive (without appealing to Church's Thesis) by
some detailed work through Recursion Theory (cf. e.g.
\cite{epstein}).

\section{Conclusions}
There are many interesting questions and suggestions for further
research at the end of \cite{yan} which motivated the research
presented in this paper; most of the questions remain unanswered as
of today. The proposed schema, i.e., the diagram of the proof of
Theorem~\ref{thm:cantor},
can be used as a criterion for testing whether an argument is
diagonal or not. What makes this argument (of the non--existence of
a fixed--point for $\alpha:D\rightarrow D$) {\em diagonal} is the diagonal
function $\triangle_B:B\rightarrow B\times B$.  In most of our
arguments we had $D={\bf 2}=\{0,1\}$ and $\alpha={\tt neg}$ by which
the proof was constructed by diagonalizing out of the function
$f:B\times B\rightarrow D$. Only in Theorem~\ref{ack} we had
$D=\mathbb{N}$ and $\alpha={\tt s}$ (the successor function) which
was used for generating a dominating function. We could have used
the diagonalizing out argument by setting $D={\bf 2}=\{0,1\}$ and
$\alpha={\tt neg}$ for the function
$\tilde{f}:\mathbb{N}\times\mathbb{N}\rightarrow{\bf 2}$, defined by
$$\tilde{f}(n,m)=\left\{\begin{array}{l} 0 \quad \textrm{if}\ \  f_n(m)=0
\\ 1  \quad \textrm{if}\ \ f_n(m)\neq 0\\
\end{array} \right.$$
Then the constructed
function $\tilde{g}:\mathbb{N}\rightarrow{\bf 2}$ by
$\tilde{g}(n)={\tt neg}(\tilde{f}(n,n))$ differs from all the
functions $f_i$'s (because $\tilde{g}(i)\neq f_i(i)$ for all $i$).
 So, this way one could construct a non--primitive recursive (but recursive) function,
 though this function does not dominate all the primitive recursive functions.

For other exciting questions and examples of theorems or paradoxes
which seem to be self--referential we refer the reader to the last
section of \cite{yan}. It will be nice to see some of those
proposals or other more phenomena fit in the above universal
diagonal schema.

{
 \paragraph{Acknowledgements.}  
The authors warmly thank  the comments and
suggestions of Professor Noson S.~Yanofsky who has wholeheartedly
encouraged the research of this paper. This is a part of the first
author's Ph.D. thesis in Tarbiat Modares University written under the supervision of the second
author who is partially supported by  a research grant (No. S/6430-1) from the University of Tabriz,  IRAN.}




\begin{thebibliography}{99}


\bibitem{boolos}{\sc George Boolos}, {Constructing Cantorian Counterexamples},
 {\it Journal of Philosophical Logic} {\bf  26:3} (1997)   237--239. 
 {\sc doi}:~\href{http://dx.doi.org/10.1023/A:1004209106100}{10.1023/A:1004209106100}




\bibitem{bc2}{\sc Ot\'{a}vio Bueno \& Mark Colyvan},
{Paradox without Satisfaction}, {\it Analysis} {\bf  63:2} (2003)  152--156. 
{\sc doi}:~\href{http://dx.doi.org/10.1111/1467-8284.00026}{10.1111/1467-8284.00026}


\bibitem{cieslinski2}{\sc Cezary Cie\'{s}li\'{n}ski \& Rafal Urbaniak},
{G\"odelizing the Yablo Sequence}, {\it Journal of Philosophical
Logic} {\bf 42:5} (2013) 679--695. 
{\sc doi}:~\href{http://dx.doi.org/10.1007/s10992-012-9244-4}{10.1007/s10992-012-9244-4}


\bibitem{epstein}
{\sc  Richard L. Epstein \&   Walter A. Carnielli}, {\it
Computability: computable functions, logic, and the foundations of
mathematics}, Advanced Reasoning Forum  (3rd ed. 2008).
{\sc isbn}:~9780981550725.
 \\   \url{http://www.advancedreasoningforum.org/computability}


\bibitem{hp}
{\sc  Petr H\'{a}jek \& and Pavel Pudl\'{a}k}, {\it Metamathematics
of First-Order Arithmetic}, Springer  (2nd. print. 1998).
{\sc isbn}:~9783540636489.
\   \url{http://projecteuclid.org/euclid.pl/1235421926}



\bibitem{gch}{\sc Akihiro Kanamori \& David Pincus},
 {``Does GCH Imply AC Locally?"}, in:
 Gabor Halasz \& Laszlo Lovasz \&  Miklos Simonovits \&  Vera T.  S\'{o}s   (eds.)
 {\it Paul Erd\"{o}s and His Mathematics II}, Bolyai Society for Mathematical Studies,
  Vol. 11, J\'{a}nos Bolyai Mathematical Society \& Springer  (2002)
  pp. 413--426.    \   \url{http://math.bu.edu/people/aki/7.pdf}



\bibitem{temporal}
{\sc  Fred Kr\"{o}ger \& Stephan  Merz}, {\it Temporal Logic and
State Systems}, Springer (2008).
{\sc isbn}:~9783540674016 \  
{\sc doi}:~\href{http://dx.doi.org/10.1007/978-3-540-68635-4}{10.1007/978-3-540-68635-4}


\bibitem{leach}{\sc Graham Leach-Krouse},
{Yablifying the Rosser Sentence}, {\it Journal of Philosophical
Logic}, to appear. 
{\sc doi}:~\href{http://dx.doi.org/10.1007/s10992-013-9291-5}{10.1007/s10992-013-9291-5}



\bibitem{pnueli}{\sc Amir Pnueli},
 {``The Temporal Logic of Programs"}, in:
 Proceedings of the 18th Annual Symposium on Foundations of Computer Science
  (SFCS'77) IEEE Computer Society, Washington DC,
  USA (1977)
  pp. 46--57. \ \  {\sc doi}:~\href{http://dx.doi.org/10.1109/SFCS.1977.32}{10.1109/SFCS.1977.32}


\bibitem{priest}{\sc Graham Priest}, {Yablo's Paradox},
{\it Analysis} {\bf  57:4} (1997)  236--242. {\sc doi}:~\href{http://dx.doi.org/10.1111/1467-8284.00081}{10.1111/1467-8284.00081}


\bibitem{yet}{\sc Natarajan Raja}, {``Yet Another Proof of Cantor's Theorem"},
in: Jean-Yves B\'{e}ziau \&  Alexandre Costa-Leite (eds.)
 {\it Dimensions of Logical Concepts}, Cole\c{c}\~{a}o CLE: Volume 54 (2009) pp. 209--217.
       \url{http://www.tcs.tifr.res.in/~raja/publications/online/dlc09.pdf}



\bibitem{yab1}{\sc Stephen Yablo},
{Truth and Reflection},
{\it Journal of Philosophical Logic} {\bf  14:3} (1985)   297--349. \\  
{\sc doi}:~\href{http://dx.doi.org/10.1007/BF00249368}{10.1007/BF00249368}


\bibitem{yab2}{\sc Stephen Yablo},
{Paradox without Self-Reference},
{\it Analysis} {\bf  53:4} (1993)  251--252. \\ 
{\sc doi}:~\href{http://dx.doi.org/10.2307/3328245}{10.2307/3328245}



\bibitem{yan}{\sc Noson  S. Yanofsky},
 {A Universal Approach to Self-Referential Paradoxes,
 Incompleteness and Fixed Points}, {\it Bulletin of  Symbolic Logic}
 {\bf  9:3} (2003) 362--386. {\sc doi}:~\href{http://dx.doi.org/10.2178/bsl/1058448677}{10.2178/bsl/1058448677}



\end{thebibliography}
\end{document}